\documentclass[a4paper]{article}
\usepackage{hyperref}
\usepackage{amsmath}
\usepackage{amsthm}
\usepackage[OT2,T1]{fontenc}
\DeclareSymbolFont{cyrletters}{OT2}{wncyr}{m}{n}
\DeclareMathSymbol{\Sha}{\mathalpha}{cyrletters}{"58}
\usepackage{CJK}
\usepackage{amssymb}
\begin{document}
	\title{Introduction to Mordell Weil Theorem}
	\author{Shenghao Li}
	\date{August 9, 2019}
	\maketitle
	\begin{abstract}
		\small This is an article about Mordell Weil Theorem. Mordell-Weil Theorem is one of the greatest theorems about ellpitic curve. In this article, I will introduce the proof of Mordell-Weil theorem and some simple ways to compute the torsion part of the group.
	\end{abstract}
    \section{Basic Properties of Elliptic Curves}
    In this part, we will introduce some basic notations and properties about elliptic curves. Some classic results won't be proved, but I will list some books where you can look them up.
    
    \newtheorem{defn}{Definition}[section]
    \begin{defn}
    	An elliptic curve is a pair (E,O), where E is a nonsingular curve of genus 1 and $O\in E$. The elliptic curve E is defined over K, written E/K, if E is defined over K(a field) as a curve and $O\in E(K)$.
    \end{defn}
    The definition above is not very clear and it's hard for us to study them. So, by Riemann-Roch theorem, we have the following equivalent definition:
    \begin{defn}
    	An elliptic curve over K can be defined as a nonsingular projective plane curve over K of the form 
    	$$Y^2Z+a_1XYZ+a_3YZ^2=X^3+a_2X^2Z+a_4XZ^2+a_6Z^3$$
    	Here O=[0,1,0] is ths basepoint.\\
    	This is called the Weierstrass equation of an elliptic curve.
    \end{defn}
    If we let $Z=0$, then we find that the point must be O. So we can assume that $Z\neq 0$, and we can get the following equation
    $$E:y^2+a_1xy+a_3y=x^3+a_2x^2+a_4x+a_6$$
    If $char(\bar{K})\neq 0$, then we can simplify the equation by completing the square. Thus replacing y by $\frac{y-a_1x-a_3}{2}$ gives an equation of the form
    $$E:y^2=4x^3+b_2x^2+2b_4x+b_6$$
    where
    \begin{align*}
    	b_2&=a_1^2+4a^2\\
    	b_4&=2a_4+a_1a_3\\
    	b_6&=a_3^2+4a_6
    \end{align*}
    We also define quantities
    \begin{align*}
    	b_8&=a_1^2a_6+4a_2a_6-a_1a_3a_4+a_2a_3^2-a_4^2\\
    	c_4&=b_2^2-24b_4\\
    	c_6&=b_2^3+36b_2b_4-216b_6\\
    	\Delta&=-b_2^2b_8-8b_4^3-27b_6^2+9b_2b_4b_6
    \end{align*}
    If further char($\bar{K}$)$\neq$ 2,3, then replacing (x,y) by $((x-3b_2)/36, y/216)$ we can get a simpler equation
    $$E:y^2=x^3-27c_4x-54c_6$$
    Thus we often use the equation $E:y^2=x^3+Ax+B$ to denote an elliptic curve, and such curves are nonsingular if and only if $\Delta\neq 0$(i.e. $-16(4A^3+27B^2)\neq 0$).
    
    Next we will intoduce one of the most important structures on elliptic curves, the group law.
    \begin{defn}
    	Let $P,Q\in E$, L the line connecting P and Q(tangent line if $P=Q$). According to the Bezout theorem, the line L and E must intersect on the third point R(may be the same as P, Q).Let L' be the line connecting R and O. Then $P\oplus Q$ is the point that such that L' intersects E at $R,O$, and $P\oplus Q$.
    \end{defn}
    Now we justify the use of the symbol $\oplus$.
    \newtheorem{pro}{Proposition}[section]
    \begin{pro}
    	The composition law($\oplus$) satifies the following properties.\\
    	(a) If a line L intersects E at the points P,Q and R, then 
    	$$(P\oplus Q)\oplus R=O$$
    	(b) $P\oplus O=P$ for all $P\in E$\\
    	(c) $P\oplus Q=Q\oplus P$ for all $P,Q\in E$\\
    	(d) Let $P\in E$. There is a point of E, denoted -P, so that
    	$$P\oplus (-P)=O$$
    	(e) Let $P,Q,R\in E$. Then 
    	$$(P\oplus Q)\oplus R=P\oplus(Q\oplus R)$$
    	In other words, the composition law makes E into an abelian group with identity element O. We further have:\\
    	(f) Suppose E is defined over K. Then
    	$$E(K)=\left\{(x,y)\in K^2:y^2+a_1xy+a_3y=x^3+a_2x^2+a_4x+a_6\right\}\cup \left\{O\right\}$$
    	is a subgroup of E.
    \end{pro}
    \begin{proof}
	    Only (e) is not trivial. One can laboriously verify the associative law case by case by checking the equations. However, we will use the Riemann Roch theorem to prove that, and also use a bit of divisors.(For definition of divisors, one can read [Sil], p.31)
    \end{proof}
    \begin{pro}
    	Let (E,O) be an elliptic curve.\\
    	(a) For every divisor $D\in Div^0(E)$ there exists a unique point $P\in E$ so that
    	$$D\sim (P)-(O)$$
    	Let
    	$$\sigma:Div^0(E)\rightarrow E$$
    	be the map given by this association.
    	(b) The map $\sigma$ is surjective.
    	(c) Let $D_1,D_2\in Div^0(E)$. Then 
    	$$\sigma(D_1)=\sigma(D_2)\qquad iff\quad D_1\sim D_2$$
    	Thus $\sigma$ induces a bijection of sets(which we also denote by $\sigma$)
    	$$\sigma:Pic^0(E)\rightarrow E$$
    	(d) The inverse to $\sigma$ is the map
    	$$\kappa:E\rightarrow Pic^0(E)$$
    	$$P\rightarrow class\ of\ (P)-(O)$$
    	(e) If E is given by a Weierstrass equation, then the composition law we mentioned above and the group law from $Pic^0(E)$ by using $\sigma$ are the same. Thus, the composition law satisfied the associative law.
    \end{pro}
    \begin{proof}
    	(a) Since E has genus 1, the Riemann-Roch theorem says that 
    	$$dim\mathcal{L}(D+(O))=1$$
    	Let $f\in \bar{K}(E)$ be a generator for $\mathcal{L}(D+(O))$. Since
    	$$div(f)\geq -D-(O)\quad and\quad deg(div(f))=0$$
    	it follows that
    	$$div(f)=-D-(O)+(P)$$
    	for some $P\in E$. Hence
    	$$D\sim (P)-(O)$$
    	To prove that P is unique, we assume that there are two points P and P' both satisfy the condition. Then we get that $P\sim P'$. So there exists $f\in \bar{K}(E)$ so that 
    	$$div(f)=(P)-(P')$$
    	Then $f\in \mathcal{L}((P'))$, and by the Riemann-Roch theorem we have $dim\mathcal{L}((P'))=1$. However we already know that the constant function is in $\mathcal{L}((P'))$, so we can get f is a constant function. Thus $P=P'$. Hence P is unique.\\
    	(b) For any $P\in E$
    	$$\sigma((P)-(O))=P$$
    	(c) Suppose $\sigma(D_1)=P_1$, $\sigma(D_2)=P_2$. Then we can get $(P_1)-(P_2)\sim D_1-D_2$. Thus $\sigma(D_1)=\sigma(D_2)$ we can imply that $D_1\sim D_2$. Also if $D_1\sim D_2$, we have $P_1\sim P_2$, so $P_1=P_2$.\\
    	(d) Directly from (b) and (c).\\
    	(e) Let E be given by a Weierstrass equation, and let $P,Q\in E$. It clearly suffices to show that 
    	$$\kappa(P+Q)=\kappa(P)+\kappa(Q)$$
    	Let 
    	$$f(X,Y,Z)=aX+bY+cZ=0$$
    	give the line L in $\mathbf{P}^2$ going through P and Q, let R be the third point of intersection of L with E, and let
    	$$f'(X,Y,Z)=a'X+b'Y+c'Z=0$$
    	be the line L in $\mathbf{P}^2$ through R and O. Then from the definition of addition on E and the fact that $Z=0$ intersects E at O with multiplicity 3, we have
    	$$div(f/Z)=(P)+(Q)+(R)-3(O)$$
    	and
    	$$div(f'/Z)=(P+Q)+(R)-2(O)$$
    	Thus
    	$$(P+Q)-(P)-(Q)+(O)=div(f'/f)$$
    	Hence
    	$$\kappa(P+Q)=\kappa(P)+\kappa(Q)$$
    \end{proof}
    \newtheorem{200}{Remark}
    \begin{200}
    	Here we will directly write out the equation of the composition law.
    	
    	Let E be an elliptic curve given by a Weierstrass equation 
    	$$E:y^2+a_1xy+a_3y=x^3+a_2x^2+a_4x+a_6$$
    	(a) Let $P_0=(x_0,y_0)\in E$. Then
    	$$-P_0=(x_0,-y_0-a_1x_0-a_3)$$
    	Now let $P_1+P_2=P_3\quad with\quad P_i=(x_i,y_i)\in E$
    	(b) If $x_1=x_2$ and $y_1+y_2+a_1x_2+a_3=0$, then
    	$$P_1+P_2=O$$
    	Otherwise, let
    	$$\lambda=\frac{y_2-y_1}{x_2-x_1},\qquad \nu=\frac{y_1x_2-y_2x_1}{x_2-x_1}\qquad if\ x_1\neq x_2$$
    	$$\lambda=\frac{3x_1^2+2a_2x_1+a_4-a_1y_1}{2y_1+a_1x_1+a_3},\qquad \nu=\frac{-x_1^3+a_4x_1+2a_6-a_3y_1}{2y_1+a_1x_1+a_3}\qquad if\ x_1= x_2$$
    	(c) $P_3=P_2+P_1$ is given by
    	$$x_3=\lambda^2+a_1\lambda-a_2-x_1-x_2$$
    	$$y_3=-(\lambda+a_1)x_3-\nu-a_3$$
    	(d) As special cases of (c), we have for $P_1\neq \pm P_2$
    	$$x(P_1+P_2)=\frac{y_2-y_1}{x_2-x_1}^2+a_1\frac{y_2-y_1}{x_2-x_1}-a_2-x_1-x_2$$
    	and the duplication formula for (x,y)$\in$ E
    	$$x([2]P)=\frac{x^4-b_4x^2-2b_6x-b_8}{4x^3+b_2x^2+2b_4x+b_6}$$
    \end{200}
    After establishing the group structure on an elliptic curve, we will now discuss a special kind of morphism between elliptic curves.
    \begin{defn}
    	Let $V_1$ and $V_2\subset \mathbf{P}^n$ be two projective varieties. A rational map from $V_1$ to $V_2$ is a map of the form
    	$$\phi:V_1\rightarrow V_2$$
    	$$\phi=[f_0,\dots,f_n]$$
    	where $f_0,\dots,f_n\in \bar{K}(V_1)$ have the property that for every point $P\in V_1$ at which $f_0,\dots,f_n$ are all defined, $\phi(P)\in V_2$.
    \end{defn}
    \begin{defn}
	    A rational map $\phi$ is regular(or defined) at P if there is a function $g\in \bar{K}(V_1)$ such that
	    (a) each $gf_i$ is regular at P
	    (b) for some i, $(gf_i)(P)\neq 0$
	    A rational map which is regular at every point in $V_1$ is called a morphism.
    \end{defn}
    Next we will state two very important results for morphisms on curves. We won't prove them and , and for those who want to see the proofs, you can check [Har, Chapter 2 Thm 6.8].
    \newtheorem{thm}{Theorem}[section]
    \begin{thm}
    	
    	For every $Q\in C_2$, we have the following equationLet $\phi:C_1\rightarrow C_2$ be a morphism of curves. Then $\phi$ is either constant or surjective.
    	
    \end{thm}
    \begin{thm}
    	Let $\phi:C_1\rightarrow C_2$ be a non-constant map of smooth curves. For all but finitely many $Q\in C_2$
    	$$\#\phi^{-1}(Q)=deg_s(\phi)$$
    \end{thm}
    Now we go back to the ellpitic curves. Because an elliptic curve contains a point O, so the map between elliptic curves should contains more imformation. Therefore we have the following definition.
    \begin{defn}
        Let $E_1$ and $E_2$ be elliptic curves. An isogeny between $E_1$ and $E_2$ is a morphism
        $$\phi:E_1\rightarrow E_2$$
        satisfying $\phi(O)=O$. $E_1$ and $E_2$ are isogenous if there is an isogeny $\phi$ between them with $\phi(E_1)=\left\{O\right\}$
    \end{defn}
    From the definition, we can clearly see that the map defined by multiplying m is an isogeny, and we use [m] to denote it.
    \begin{pro}
    	Let E/K be an elliptic curve and let $m\in \mathbf{Z}$, $m\neq 0$. Then the multiplication by m map
    	$$[m]:E\rightarrow E$$
    	is non-constant.
    \end{pro}
    \begin{proof}
    	We start by showing that $[2]\neq [0]$. From the duplication formula, if a point P=(x,y)$\in$ E has order 2, then it must satisfy
    	$$4x^3+b_2x^2+2b_4x+b_6=0$$
    	which only has finitely many solutions. Therefore $[2]\neq [0]$. Now, using the fact that $[mn]=[m][n]$, we are reduced to considering the case of odd m.
    	
    	Using the long division, one can easily find out that the polynomials 
    	$$4x^3+b_2x^2+2b_4x+b_6$$
    	does not divide
    	$$x^4-b_4x^3-2b_6x-b_8$$
    	(If it does, then $\Delta$=0, contradiction). Hence we can find an $x_0\in \bar{K}$ so that the former vanishes to a higher order at $x=x_0$ than the latter. Choosing $y_0\in \bar{K}$ so that $P_0=(x_0,y_0)\in E$, the doubling formula implies that [2]$P_0$=O.In other words, we have shown that E has a non-trivial point of order 2. But then for m odd
    	$$[m]P_0=P_0\neq O$$
    	so clearly $[m]\neq [0]$.
    \end{proof}
    \begin{thm}
    	Consider the isogeny $[m]:E\rightarrow E$. For every $Q\in E_2$
    	$$\#[m]^{-1}(Q)=deg_s[m]$$
    \end{thm}
    \begin{proof}
    	From Theorem 1.2 we know that
    	$$\#[m]^{-1}(Q)=deg_s[m]$$
    	for all but finitely many $Q\in E_2$. But for any $P,P'\in E_1$, if $[m]P=[m]P'$, then $P-P'\in [m]^{-1}(O)$. Thus for every $Q\in E_2$, $[m]^{-1}(Q)$ is a coset of $[m]^{-1}(O)$. So for all Q, we have 
    	$$\#[m]^{-1}(Q)=deg_s[m]$$
    \end{proof}
    By now, we have introduced some important properties of elliptic curves, and next we will introduce the Mordell Weil theorem.
    \section{Mordell Weil Theorem}
    \newtheorem{theorem}{Theorem}[section]
    \begin{theorem}
    \textup{(Mordell-Weil)} Let E be an elliptic curve defined over a number field K. The group E(K) is a finitely generated Abelian group
    \end{theorem}
    The proof is given in two parts: The first part is called the Weak Mordell-Weil Theorem, which proves that $E(K)/nE(K)$ is finite, and the second part uses height function to prove $E(K)$ is finitely generated.
    \subsection{Weak Mordell-Weil Theorem}
    In this section, we will give two proofs of the Weak Mordell-Weil Theorem.
    \newtheorem{1}{Theorem}[subsection]
    \begin{1}
    	\textup{(Weak Mordell-Weil)} Let E be an elliptic curve defined over a number field K. Then E(K)/mE(K) is finite for any n $\geq$ 2.
    \end{1}
    The first proof, given by Silverman, is based on theories  about field extension.
    \newtheorem{2}{Lemma}[subsection]
    \begin{2}
    	Let L/K be a finite Galois extension. If E(L)/mE(L) is finite, then E(K)/mE(K) is finite.
    \end{2}
    \begin{proof}
    	Let $\Phi$ be the kernel of the natural map $E(K)/nE(K) \rightarrow E(L)/nE(L)$. Therefore,$$\Phi=(E(K)\cap mE(L))/mE(K)$$
    	and for each $P$ (mod $mE(K)$) in $\Phi$, we can choose a point $Q_p \in E(L)$ with $[m]Q_p=P$. Having done this, we define a map of sets
    	$$\lambda_p:G_{L/K} \rightarrow E[m], \qquad      \lambda_p(\sigma)=Q_p^\sigma-Q_p$$Here $Q_p$ is fixed for each $P$.\\
    	We notice that
    	$$\lambda_p(\sigma)=[m](Q_p^\sigma-Q_p)=[m]Q_p^\sigma-[m]Q_p=0$$
    	So $\lambda_p(\sigma)$ is in $E[m]$.\\
    	Suppose that $\lambda_p=\lambda_{p'}$ for two points $P,P' \in E(K)\cap mE(L)$. Then we have
    	$$(Q_p-Q_{p'})^\sigma=Q_p-Q_{p'} \qquad for \ all \ \sigma \in G_{L/K}$$
    	so $Q_p-Q_{p'} \in E(K)$. Therefore
    	$$P-P'=[m](Q_p-Q_{p'})\in mE(K)\Leftrightarrow P\equiv P'\pmod {mE(K)}$$
    	So the map
    	$$\Phi\rightarrow Map(G_{L/K},\ E[m]),\qquad P\rightarrow\lambda_p$$
    	is an injection. But $G_{L/K}$ and $E[m]$ are finite sets, so $\Phi$ is a finite set.\\
    	Finally, the exact sequence
    	$$0\rightarrow\Phi\rightarrow E(K)/mE(K)\rightarrow E(L)/mE(L)$$
    	implies that $E(K)/mE(K)$ is finite(because it is between two finite sets).
    \end{proof}
    In view of the lemma above, we can enlarge the number field K and suppose that $E[m]\subset E(K)$ (because $E[m]$ is finite). We will assume this is true for the remainder of this section.\\
    The next step we will do is to translate the question into a question about a certain field extension of K.
    \newtheorem{3}{Definiton}[subsection]
    \begin{3}
    	The Kummer pairing
    	$$\kappa:E(K)\times G_{\overline{K}/K}\rightarrow E[m]$$
    	is defined as follows. Let $P\in E(K)$, and choose any $Q\in E(\overline{K})$ satisfying $[m]Q=P$. Then
    	$$\kappa(P,\ \sigma)=Q^\sigma-Q$$
    \end{3}
    Actually, from the definition we can see that it is similar to the definition of $\lambda_p$. It is well-defined because $E[m]\subset E(K)$.
    \begin{1}
    	(a) The Kummer pairing is bilinear.\\
    	(b) The kernel of the Kummer pairing on the left is $mE(K)$.\\
    	(c) The kernel of the Kummer pairing on the right is $G_{\overline{K}/L}$, where
    	$$L=K([m]^{-1}E(K))$$
    	is the compositum of all fields $K(Q)$ as $Q$ ranges over the points of $E(\overline{K})$ satisfying $[m]Q\in E(K)$.\\
    	Hence the Kummer pairing induces a perfect bilinear pairing
    	$$E(K)/mE(K)\times G_{L/K}\rightarrow E[m]$$
    \end{1}
    \begin{proof}
    	(a) The linearity of P is trivial. For $\sigma$, let $\sigma,\tau\in G_{\overline{K}/K}$. Then
    	$$\kappa(P, \ \sigma\tau)=Q^{\sigma\tau}-Q=(Q^\sigma-Q)^\tau+Q^\tau-Q=\kappa(P,\ \sigma)^\tau+\kappa(P, \ \tau)$$
    	However, $\kappa(P,\ \sigma)\in E[m]\subset E(K)$, so it is fixed by $\tau$. Therefore, 
    	$$\kappa(P, \ \sigma\tau)=\kappa(P,\ \sigma)+\kappa(P, \ \tau)$$
    	(b) Suppose $\kappa(P, \ \sigma)=0$ for all $\sigma\in G_{\overline{K}/K}$. Then we have $Q^\sigma=Q$ for all $\sigma\in G_{\overline{K}/K}$.Therefore, $Q\in E(K)$ and $P=[m]Q\in mE(K)$
    	And if $P\in mE(k)$, it is obvious that $\kappa(P, \ \sigma)=0$ for all $\sigma\in G_{\overline{K}/K}$. Therefore, the kernel on the left is $mE(K)$.\\
    	(c) Suppose $\kappa(P, \ \sigma)=0$ for all $P\in E(K)$, then $Q^\sigma-Q=0$ for all $Q$ satisfying $[m]Q\in E(K)$. But $L$ is the compositum of $K(Q)$ over all such $Q$, so $\sigma$ fixes $L$. Hence $\sigma\in G_{\overline{K}/L}$. Conversely, if $\sigma\in G_{\overline{K}/L}$, then we have
    	$$\kappa(P, \ \sigma)=Q^\sigma-Q=0$$
    	since $Q\in E(L)$ from the definition. Thus the kernel on the right is $ G_{\overline{K}/L}$.\\
    	Finally, for the last statement of the theorem, we firstly claim that $L/K$ is Galois because it is normal from the definition($[m]Q'=[m]Q\in E(K)$ if $Q'$ is a conjugate of $Q$). Since $L/K$ is Galois, we have
    	$$G_{\overline{K}/K}/G_{\overline{K}/L}=G_{L/K}$$
    	Thus it is a perfect bilinear pairing.   	
    \end{proof}
    From Theorem 2.1.2 we can see that if we can prove $L$ is a finite extension, or in other words, $G_{L/K}$ is finite, then the group $E(K)/mE(K)$ is finite. So the next step is to analyze this extension.
    \begin{1}
    	Let $L$ be the field defined in Theorem 2.1.2.\\
    	(a) $L/K$ is an abelian extension of exponent m.(I.e. $G_{L/K}$ is abelian and every element has order dividing m.)\\
    	(b) Let
    	$$S=\left\{v\in M_K^0:\ E\ has\ bad\ reduction\ at\ v\right\}\cup \left\{v\in M_K^0:v(m)\neq0\right\}\cup M_K^\infty$$
        Then $L/K$ is unramified outside S.
    \end{1}
    \begin{proof}
    	(a) This follows immediately from the last statement of Theorem 2.2, which implies that there is an injection
    	$$G_{L/K}\rightarrow Hom(E(K),\ E[m])$$
    	$$\sigma\rightarrow\kappa(\cdot,\ \sigma)$$\\
    	(b) Let $v\in M_K$ with $v\notin S$. Choose an arbitrary element $Q$ in $m^{-1}E(K)$, and the only thing we have to show is that $K'=K(Q)$ is unramified at v, because $L$ is the compositum of all such $K'$. Let $v'\in M_{K'}$ be a place of $K'$ such that $v\mid v'$, and let $k_{v'}/k_v$ be the corresponding extension of residue fields. Since $E$ has good reduction at $v$, $E$ also has good reduction at $v'$(because the discriminants are the same). Thus we have the usual reduction map
    	$$E(K')\rightarrow \tilde{E_{v'}}(k_{v'}')  $$
    	Now let $I_{v'/v}\subset G_{K'/K}$ be the inertia group for $v'/v$, and let $\sigma\in I_{v'/v}$. By definition of inertia, $\sigma$ acts trivially on $\tilde{E_{v'}}(k_{v'}')$, so
    	$$\tilde{Q^\sigma-Q}=\tilde{Q^\sigma}-\tilde{Q}=\tilde{0}$$
    	On the other hand, $Q^\sigma-Q\in E(K)[m]$, so $Q^\sigma=Q$. Thus $Q$ is fixed by all elements of $I_{v'/v}$, which implies that the action of inertia group on $K'$ is trivial. Hence $K'$ is unramified over $K$ at $v'$. 
    \end{proof}
    Next we will prove that all field extensions $L/K$ satifying the condition in Theorem 2.1.3 must be a finite field extension.
    \begin{1}
    	Let K be a number field, $S\subset M_K$ a finite set of places containing $M_K^\infty$, and $m\geq2$ an integer. Let $L/K$ be the maximal abelian extension of $K$ having exponent m which is unramified outside of S. Then $L/K$ is a finite extension.
    \end{1}
    \begin{proof}
    	First, we can assume that K contains the $m^th-roots$ of unity $\mu_m$. That is because if K doesn't contain it, we can choose $K'=K(\mu_m)$ and $LK'/K'$ is also an abelian extension of exponent m unramified at $S'$, where $S'$ is the set of places of $K'$ lying over S. And if $LK'/K'$ is finite, $L/K$ is also finite. So we can assume that K contains the $m^th-roots$ of unity $\mu_m$.
    	
    	Furthermore, we may increase the set S, because this can only make the field extension larger. Using the fact that the class number of K is finite, we can thus add a finite number of elements to S so that the ring of S-integers
    	$$R_s=\left\{a\in K:v(a)\geq 0\ for\ all\ v\in M_K, v\notin S\right\}$$
    	is a principle ideal domain. We may also enlarge S so that $v(m)=0$ for all $v\notin S$.
    	
    	Next, according to Kummer theory, we know that L is the largest subfield of $K(\sqrt[m]{a}:a\in K) $ which is unramified outside S.
    	
    	Let $v\in M_K$, $v\notin S$. Looking at the equation
    	$$X^m-a=0$$
    	over local field $K_v$, and remembering that $v(m)=0$, it is clear that $K_v(\sqrt[m]{a})/K_v$ is unramified iff 
    	$$ord_v(a) \equiv 0\pmod m$$
    	Therefore, $L=K(\sqrt[m]{a}:a\in T_S)$, where
    	$$T_S=\left\{a\in K^*/(K^*)^m:ord_v(a)\equiv 0\pmod m\right\}$$
    	Hence if we can prove that $$T_S$$ is a finite group, we can see that L is a finite extension over K.
    	To prove $T_S$ is finite, we first consider the natural map
    	$$R_S^*\rightarrow T_S$$
    	We claim that the map is surjective. To see this, suppose $a\in K^*$ represents an element of $T_S$. Then the ideal $aR_S$ is the $m^{th}$-power of an ideal in $R_S$, since the prime ideals of $R_S$ correspond to the valuations $v\notin S$. Since $R_S$ is a principle ideal domain, we can find $b\in K^*$ s.t. $aR_S=b^mR_S$, which means that 
    	$$a=ub^m$$
    	for $u\in R_S^*$. Then u and a give the same element of $T_S$, showing that the map is surjective. Now the kernel of the map certainly contains $(R_S^*)^m$, so we have a surjection
    	$$R_S^*/(R_S^*)^m\rightarrow T_S$$
    	According to Dirichlet's unit theorem, which shows that the group of units is finitely generated, we know that $R_S^*/(R_S^*)^m$ is a finite group. Thus $T_S$ is a finite group, and the proof is completed.
    \end{proof}
    From Theorem 2.1.4 we can see that $L/K$ is a finite galois extension, so $G(L/K)$ is finite. Therefore, the group $E(k)/mE(K)$ is finite, and we have the Weak Mordell Weil theorem correct.

    Next we will use cohomology to prove the Weak Mordell Weil theorem. First we will introduce group cohomology.
    \begin{3}
    	Let G be a finite group acting on an abelian group M. We define
    	$$H^0(G, M)=M^G=\left\{m\in M\mid \sigma m=m,\ all\ \sigma\in G\right\}$$
    	A crossed homomorphism is a map $f:G\rightarrow M$ such that
    	$$f(\sigma\tau)=f(\sigma)+\sigma f(\tau)\qquad all\ \sigma,\tau\in G$$
    	and a crossed homomorphism is said to be principal if given an $m\in M$
    	$$f(\sigma)=\sigma m-m,\qquad all\ \sigma\in G$$
    	Next we define
    	$$H^1(G, M)=\frac{\left\{crossed\ homomorphisms\right\}}{\left\{principle\ crossed homomorphisms\right\}}$$
    \end{3} 
    We then state the most important and basic properties of cohomology.
    \newtheorem{4}{Proposition}[subsection]
    \begin{4}
    	For any exact sequence of G-modules
    	$$0\rightarrow M\rightarrow N\rightarrow P\rightarrow 0$$
    	there is a canonical exact sequence
    	$$0\rightarrow H^0(G, M)\rightarrow H^0(G, N)\rightarrow H^0(G, P)\xrightarrow{\delta} H^1(G, M)\rightarrow H^1(G, N)\rightarrow H^1(G, P)$$
    \end{4} 
    However, we want to solve problem about field extension, which might be infinite, so we have to develop the theory about cohomology of infinite Galois group.
    \begin{3}
    	Let K be a perfect field, $\bar{K}$ its algebraic closure, and let 
    	$$G:=Gal(\bar{K}/K)=G_K$$
    	be its Galois group. Then we can dress G in the Krull topology: a subgroup is open if it fixes a finite extension of K. Thus all these subgroups form a base of $1_G$. And thus they form a base for every point $g\in G$. So we can a topology on G.\\
    	Next a G-module M is said to be discrete if the map $G\times M\rightarrow M$ is continuous relative to the discrete topology on M and the Krull topology on G. This is equivalent to requiring that every element of M is fixed by the subgroup of G fixing some finite extension of K.  
    \end{3}
 	For a dicrete $G-module\ M$, every principle crossed homomorphism $f:G\rightarrow M$ is continuous. That is because every element of $M$ is fixed by an open normal subgroup of $G$.
    \begin{3}
    	$$H^1(G, M)=\frac{\left\{continuous\ crossed\ homomorphisms\right\}}{\left\{principle\ crossed homomorphisms\right\}}$$
    \end{3} 
    And still, we have Theorem correct.\\
    \\
    Also, we have the short exact sequence
    $$0\rightarrow E(\bar{\mathbf{Q}})[m]\rightarrow E(\bar{\mathbf{Q}})\xrightarrow{m} E(\bar{\mathbf{Q}})\rightarrow 0$$
    Therefore, we can get the long exact sequence
    $$0\rightarrow E(\mathbf{Q})[m]\rightarrow E(\mathbf{Q})\xrightarrow{m} E(\mathbf{Q})\xrightarrow{\delta} H^1(\mathbf{Q}, E[m])\rightarrow H^1(\mathbf{Q}, E)\xrightarrow{m} H^1(\mathbf{Q}, E)$$
    From this, we can get another short exact sequence
    $$0\rightarrow E(\mathbf{Q})/mE(\mathbf{Q})\xrightarrow{\delta} H^1(\mathbf{Q}, E[m])\rightarrow H^1(\mathbf{Q}, E)[m]\rightarrow 0$$
    Since $\delta$ is an injection here, if we can prove the group $H^1(\mathbf{Q}, E(\mathbf{Q})[m])$ is finite, then we can prove the Weak Mordell-Weil Theorem. However, this might not be true. So we will use the local field $\mathbf{Q}_p$ to solve the problem.
    
    First, we choose the algebraic closure $\bar{\mathbf{Q}}$ for $\mathbf{Q}$, and $\bar{\mathbf{Q}_p}$ for $\mathbf{Q}_p$. The embedding $\mathbf{Q}\hookrightarrow \mathbf{Q}_p$ extends to an embedding $\bar{\mathbf{Q}}\hookrightarrow \bar{\mathbf{Q}_p}$. Moreover, the action of $Gal(\bar{\mathbf{Q}_p}/\mathbf{Q}_p)$ on $\bar{\mathbf{Q}}\subset \bar{\mathbf{Q}_p}$ defines a homomorphism $\psi:G_{\mathbf{Q}_p}\rightarrow G_Q$ by restriction of the Galois action.
    
    Therefore, a crossed homomorphism $f:G_{\mathbf{Q}}\rightarrow E(\bar{\mathbf{Q}})$ defines a crossed homomorphism $\tilde{f}:G_{\bar{\mathbf{Q}}_p}\rightarrow E(\bar{\mathbf{Q}}_p)$ by composition $\tilde{f}=f\circ \psi$. To check this is well defined, for any $\sigma,\tau\in G_{\mathbf{Q}_p}$,
    $$\tilde{f}(\sigma\tau)=f(\psi(\sigma\tau))=f(\psi(\sigma))+\psi(\sigma)f(\psi(\tau))=f(\psi(\sigma))+\sigma f(\psi(\tau))=\tilde{f}(\sigma)+\sigma \tilde{f}(\tau)$$
    And also if $f$ is a principle crossed homomorphism, then $\tilde{f}$ is also principle. Thus we can get a map $\phi: H^1(\mathbf{Q}, E)\rightarrow H^1(\mathbf{Q}_p, E)$ by taking $f$ to $\tilde{f}$.
    
    We can get the following commutative diagram
    \begin{align*}
    	0\rightarrow E(\mathbf{Q})&/mE(\mathbf{Q})\xrightarrow{\delta} H^1(\mathbf{Q}, E[m])\rightarrow H^1(\mathbf{Q}, E)[m]\rightarrow 0\\
    	&\downarrow\qquad\qquad\qquad\qquad\downarrow\qquad\qquad\qquad\qquad\downarrow\\
    	0\rightarrow E(\mathbf{Q}_p)&/mE(\mathbf{Q}_p)\xrightarrow{\delta} H^1(\mathbf{Q}_p, E[m])\rightarrow H^1(\mathbf{Q}_p, E)[m]\rightarrow 0
    \end{align*}
     where the top and bottom lines are exact and the vertical maps are embedding.
     
     Next we reach a crucial argument. If some $\gamma\in H^1(\mathbf{Q}, E[m])$ comes from the class of an element of $E(\mathbf{Q})$, then its image $\gamma_p\in H^1(\mathbf{Q}_p, E[m])$ arises from an element of $E(\mathbf{Q}_p)$. We want to quantify those $\gamma$ whose local versions $\gamma_p$ comes from $E(\mathbf{Q}_p)$ and all those $\gamma$ which vanish locally. 
     
     Here comes two definitions that we will mainly talk about.
     \begin{3}
     	The $n-Selmer group$ is defined by
     	\begin{align*}
     		S^{(n)}(E/\mathbf{Q}):&=\left\{\gamma\in H^1(\mathbf{Q}, E[n])\mid \forall p,\ \gamma_p \ comes\ from\ E(\mathbf{Q}_p)\right\}\\
     		&=ker(H^1(\mathbf{Q}, E[n])\rightarrow \prod_{p\ prime} H^1(\mathbf{Q}_p, E))
     	\end{align*}	
     \end{3} 
     \begin{3}
 	    The $Tate-Shafarevich$ group is defined by
 	    $$\Sha(E/\mathbf{Q})=ker(H^1(\mathbf{Q}, E)\rightarrow \prod_{p\ prime} H^1(\mathbf{Q}_p, E)	)$$
     \end{3} 
     And we need the following lemma, which is easy to prove.
     \begin{2}
     	For any chain of modules $A\xrightarrow{\alpha}B\xrightarrow{\beta}C$, we can get a long exact sequence
     	$$0\rightarrow ker(\alpha)\rightarrow ker(\beta\alpha)\rightarrow ker(\beta)\rightarrow coker(\alpha)\rightarrow coker(\beta\alpha)\rightarrow coker(\beta)\rightarrow 0$$
     \end{2}
     We won't prove this because all the maps are natural.
     
     If we apply the lemma to the maps
     $$H^1(\mathbf{Q}, E[n])\rightarrow H^1(\mathbf{Q}, E)[n]\rightarrow \prod_{p\ prime} H^1(\mathbf{Q}_p, E)[n])$$,
     we obtain the fundamental exact sequence
     $$0\rightarrow E(\mathbf{Q})/nE(\mathbf{Q})\rightarrow S^{(n)}E/\mathbf{Q}\rightarrow \Sha(E/\mathbf{Q})[n]\rightarrow 0$$
     We shall prove $E(\mathbf{Q})/nE(\mathbf{Q})$ to be finite by showing that $S^{(n)}E/\mathbf{Q}$ is finite.
     
     First we will prove the Selmer group is finite in a special case.
     \begin{2}
     	If all the pointos of order 2 on an elliptic curve given by the Weierstrass equation
     	$$Y^2Z+a_1XYZ+a_3YZ^2=X^3+a_2X^2Z+a_4XZ^2+a_6Z^3$$
     	have coordinates in $\mathbf{Q}$, then the Selmer group $S^{(2)}(E/\mathbf{Q})$ is finite.
     \end{2}
     \begin{proof}
     	Since they all have coordinates in $\mathbf{Q}$, we can imply that
     	$$E(\bar{\mathbf{Q}})[2]=E(\mathbf{Q})[2]\cong (\mathbf{Z}/2\mathbf{Z})\times (\mathbf{Z}/2\mathbf{Z})$$
     	(One can check [Sil, Cor 6.4(b)] for the proof)
     	
     	And the group $Gal(\bar{\mathbf{Q}}/\mathbf{Q})$ acts trivially on E[2]. Thus we have
     	$$H^1(\mathbf{Q}, E[2])\cong (\mathbf{Q}^\times/\mathbf{Q}^{\times2})^2$$
     	(One can get this by using the long exact sequence of cohomology on the short exact sequence
     	$$1\rightarrow \mathbf{Z}/2\mathbf{Z}\rightarrow \mathbf{Q}^\times \xrightarrow{2} \mathbf{Q}^\times\rightarrow 1$$)

     	Let $\gamma\in S^{(2)}(E/\mathbf{Q})\subset H^1(\mathbf{Q}, E[2])$. For each prime $p_0$ not dividing 2$\Delta$, there exists a finite unramified extension K of $\mathbf{Q}_{p_0}$ such that $\gamma$ maps to zero under the vertical arrows:
     	\begin{align*}
     		H^1(&\mathbf{Q}, E[2])\xrightarrow{\cong} (\mathbf{Q}^\times/\mathbf{Q}^{\times2})^2\\
     		&\downarrow \qquad \qquad \qquad \downarrow\\
     		H^1(&K, E[2])\xrightarrow{\cong} (K^\times/K^{\times2})^2
     	\end{align*}
     	We choose a representative element $((-1)^{\varepsilon(\infty)}\prod_{p}p^{\varepsilon(p)},(-1)^{\varepsilon'(\infty)}\prod_{p}p^{\varepsilon'(p)})\in (\mathbf{Q}^\times/\mathbf{Q}^{\times2})^2$ for $\gamma$. Here each $\varepsilon$ or $\varepsilon'$ is either 0 or -1. Therefore we can see that
     	$$ord_{p_0}((-1)^{\varepsilon(\infty)}\prod_{p}p^{\varepsilon(p)})=\varepsilon(p_0)$$
     	and so if $(-1)^{\varepsilon(\infty)}\prod_{p}p^{\varepsilon(p)}$ is a square in K, then $\varepsilon(p_0)=0$. Therefore the only p that can occur in the factorizations are those dividing 2$\Delta$, which allows only finitely many possibilities for $\gamma$.
     \end{proof}
     After proving the special case, we will now turn to prove the general case.
     \begin{1}
     	The Selmer group $S^{(n)}(E/\mathbf{Q})$ is finite.
     \end{1}
     \begin{proof}
     	Actually, instead of proving $S^{(n)}(E/\mathbf{Q})$ is finite, we want to prove that
     	$$S^{(n)}(E/L):=ker(H^1(L, E[n])\rightarrow \prod_{v\in M_K} H^1(\mathbf{Q}_p, E))$$
     	is finite for any suitably large L. And according to the next lemma, we will show that if it is correct for L, then it is correct for $\mathbf{Q}$.
     	\begin{2}
     		For any finite Galois extension L of $\mathbf{Q}$ and integer $n\geq 1$, the kernel of 
     		$$S^{(n)}(E/\mathbf{Q})\rightarrow S^{(n)}(E/L)$$
     		is finite
     	\end{2}
        \begin{proof}
        	Since $S^{(n)}(E/\mathbf{Q})$ and $S^{(n)}(E/L)$ are subgroups of $H^1(\mathbf{Q},E[n])$ and $H^1(L,E[n])$ respectively, it suffices to prove that the kernel of 
        	$$H^1(\mathbf{Q},E[n])\rightarrow H^1(L,E[n])$$
        	is finite. However, we can easily verify that the kernel of the map is $H^1(Gal(L/\mathbf{Q}), E(L)[n])$, which is finite because both Gal(L/$\mathbf{Q}$) and E(L)[n] are finite.
        \end{proof}
        Here we still need some preparations from algebraic number theory.
        \begin{2}
        	When T is a finite set of prime ideals in L, the groups $U_T$ and $C_T$ defined by the exactness of the sequence
        	$$0\rightarrow U_T\rightarrow L^\times \xrightarrow{a\rightarrow(ord_\mathfrak{p}(a))} \bigoplus_{\mathfrak{p}\notin T}\mathbf{Z}\rightarrow C_T\rightarrow0$$
        	are, respectively, finitely generated and finite.
        \end{2}
        \begin{proof}
        	First, let's consider the kernel of the map 
        	$$f:L^\times\rightarrow \bigoplus_{\mathfrak{p}}\mathbf{Z}$$
        	An element a is $kerf$ iff $ord_\mathfrak{p}(a)=0$ for all $\mathfrak{p}$, thus a is in the kernel iff it is a unit of $O_L$. And the cokernel of f is obviously finite due to the finiteness of the class number. Hence we get an exact sequence 
        	$$0\rightarrow U\rightarrow L^\times \xrightarrow{a\rightarrow(ord_\mathfrak{p}(a))} \bigoplus_{\mathfrak{p}}\mathbf{Z}\rightarrow C\rightarrow0$$
        	where U is the unit group of $O_L$, and C is the ideal class group. So U is finitely generated due to the Dedekind Unit theorem and C is a finite group.
        	
        	Next, use the kernel-cokernel exact sequence of
        	$$L^\times \rightarrow \bigoplus_{\mathfrak{p}}\mathbf{Z}\rightarrow \bigoplus_{\mathfrak{p}\notin T}\mathbf{Z}$$
        	is an exact sequence
        	$$0\rightarrow U\rightarrow U_T\rightarrow \bigoplus_{\mathfrak{p}\in T}\mathbf{Z}\rightarrow C\rightarrow C_T\rightarrow o$$
        	Thus we can see that $U_T$ and $C_T$ are finitely generated and finite recpectively.
        \end{proof}
        Now we come back to the proof of the theorem. 
        
        Let's review the proof of the special case. Actually, we can see that the proof used the following facts:\\
        (a) $\mathbf{Q}$ contains a primitive square root of 1\\
        (b) The points of order 2 all have coordinates in $\mathbf{Q}$\\
        (c) For any finite set T of prime numbers, the kernel of 
        $$r\rightarrow(ord_p(r)\pmod 2):\mathbf{Q}^\times/\mathbf{Q}^{\times2}\rightarrow \bigoplus_{p\in T}\mathbf{Z}/2\mathbf{Z}$$
        is finite.
        
        Therefore, according to the above discussion, what we have to do is to prove the following lemma and the proof will be completed.
        \begin{2}
        	Assume that L contains the $n^{th}$-unity root. For any finite subset T of $M_L$ containing $M_K^\infty$, let N be the kernel of
        	$$a\rightarrow(ord_\mathfrak{p}(a)\pmod n):L^\times/L^{\times n}\rightarrow \bigoplus_{\mathfrak{p}\in T}\mathbf{Z}/n\mathbf{Z}$$
        	Then there is an exact sequence
        	$$0\rightarrow U_T/U_T^n\rightarrow N\rightarrow C_T[n]$$
        	Therefore N is a finite group.
        \end{2}
        \begin{proof}
        	This can be proved by a diagram chase in
        	\begin{align*}
        		0\rightarrow U_T\rightarrow &L^\times \rightarrow \bigoplus_{\mathfrak{p}\notin T}\mathbf{Z}\rightarrow C_T\rightarrow 0\\
        		\downarrow n\quad &\downarrow n\qquad\downarrow n\quad\ \downarrow n\quad\\
        		0\rightarrow U_T\rightarrow &L^\times \rightarrow \bigoplus_{\mathfrak{p}\notin T}\mathbf{Z}\rightarrow C_T\rightarrow 0\\
        		&\downarrow \qquad\quad \downarrow\\
        		L^\times&/L^{\times n}\rightarrow \bigoplus_{\mathfrak{p}\in T}\mathbf{Z}/n\mathbf{Z}
        	\end{align*}
        \end{proof}
        Since we have the lemma correct, we have the theorem correct, and the proof is cmpleted.
     \end{proof}
     Actually, we can see that the proof above can prove that the Selmer group $S^{(n)}(E/K)$ is finite for any number field K. Therefore we prove the Weak Mordell Weil theorem by using cohomology.
     \subsection{The Descent Procedure and Height Function on $\mathbf{Q}$}
     In this section, we will prove the Mordell Weil theorem on $\mathbf{Q}$.
     \begin{4}
     	(Descent theorem) Let A be an abelian group. Suppose there is a 'height' funtion 
     	$$h:A\rightarrow\mathbf{R}$$
     	with the following three properties:\\
     	($\romannumeral1$) Let $Q\in A$. There is a constant $C_1$ depending on A and Q, so that for all $P\in A$,
     	$$h(P+Q)\leq 2h(P)+C_1$$
     	($\romannumeral2$) There is an integer $m\geq2$ and a constant $C_2$, depending on A, so that for all $P\in A$,
     	$$h(mP)\geq m^2h(P)-C_2$$
     	($\romannumeral3$) For every constant $C_3$,
     	$$\left\{P\in A:h(P)\leq C_3\right\}$$
     	is a finite set.\\
     	Suppose further that for the integer m in ($\romannumeral2$), the quotient group $A/mA$ is finite. Then A is finitely generated.
     \end{4}
     \begin{proof}
     	Choose elements $Q_1,\dots, Q_r\in A$ to represent the finitely many cosets in A/mA. The idea is to show that by substracting an appropriate linear combination of $Q_1,\dots, Q_r$ from P, we will be able to make the height of the resulting point less than a constant which is independent of P. Then the $Q_1,\dots, Q_r$ and the finitely many points with height less than this constant will generate A.
     	
     	Write
     	$$P=mP_1+Q_{i_1}\qquad for\ some\ 1\leq i_1\leq r$$
     	Continuing in this fashion,
     		$$P_1=mP_2+Q_{i_2}$$
     		$$.$$
     		$$.$$
     		$$.$$
     		$$P_{n-1}=mP_n+Q_{i_n}$$
     	Now for any j, we have
     	\begin{align*}
     		h(P_j)&\leq \frac{1}{m^2}[h(mP_j)+C_2]\qquad from (\romannumeral3)\\
     		&=\frac{1}{m^2}[h(P_{j-1}-Q_{i_j})+C_2]\\
     		&\leq\frac{1}{m^2}[2h(mP_{j-1})+C_1'+C_2]\qquad from (\romannumeral1)
     	\end{align*}
     	
        where we take $C_1'$ to be the maximum of the constants from ($\romannumeral1$) for $Q=-Q_i,\ 1\leq i\leq r$. Note that $C_1'$ and $C_2$ do not depend on P.
        Now use the above inequality repeatedly, starting from $P_n$ and working back to P. This yields
        \begin{align}
        	h(P_n)&\leq(\frac{2}{m^2})^nh(P)+[\frac{1}{m^2}+\frac{2}{m^4}+\frac{4}{m^6}+\dots+\frac{2^{n-1}}{m^{2n}}](C_1'+C_2)\\
        	&<(\frac{2}{m^2})^nh(P)+\frac{C_1'+C_2}{m^2-2}\\
        	&\leq 2^{-n}h(P)+(C_1'+C_2)/2
        \end{align}
        It follows that by taking n sufficiently large, we will have
        $$h(P_n)\leq 1+(C_1'+C_2)/2$$
        Since
        $$P=m^nP_n+\Sigma_{j=1}^n m^{j-1}Q_[i_j]$$
        it follows that every $P\in A$ is a linear combination of the points in the set
        $$\left\{Q_1,\dots, Q_r\right\}\cup\left\{Q\in A:h(Q)\leq 1+(C_1'+C_2)/2\right\}$$
        And from the third property, this is a finite set, which proves that A is finitely generated.
        \end{proof}
     Therefore, to solve the problem, all we have to do is to find a height function on $E(K)$ satisfying the three properties. First let's talk about how to define a height function on $E(\mathbf{Q})$.
     
     Fix a Weierstrass equation for $E/\mathbf{Q}$ of the form
     $$E:y^2=x^3+Ax+B$$
     with $A,B\in \mathbf{Z}$.
     
     \begin{3}
    	 Let $t\in \mathbf{Q}$ and write $t=p/q$ as a fraction in lowest terms. The height of t, denoted H(t), is defined by
    	 $$H(t)=max\left\{|p|,|q|\right\}$$
     \end{3}
     \begin{3}
    	 The height on $E(\mathbf{Q})$(relative to the given Weierstrass equation) is the function
    	 $$h_x:E(\mathbf{Q})\rightarrow \mathbf{R}$$
    	 $$h_x(P)=\left\{ \begin{array}{rcl}
    	 logH(x(P)) & \mbox{if}
    	 & P\neq O \\ 0 & \mbox{if} & P=O \end{array}\right.$$
    \end{3}
    We want to prove that the height function defined above has the three properties. Therefore we should prove the following lemma
    \begin{2}
    	(a) Let $P_0\in E(\mathbf{Q})$. There is a constant $C_1$, depending on $P_0$, A, B, so that for all $P\in E(\mathbf{Q})$,
    	$$h_x(P+P_0)\leq 2h_x(P)+C_1$$
    	(b) There is a constant $C_2$, depending on A, B, so that for all $P\in E(\mathbf{Q})$,
    	$$h_x([2]P)\geq 4h_x(P)-C_2$$
    	(c) For every constant $C_3$, the set
    	$$\left\{P\in E(\mathbf{Q}):h_x(P)\geq C_3\right\}$$
    	is finite.
    \end{2}
    \begin{proof}
    	Taking $C_1>max\left\{h_x(P_0),h_x([2]P_0)\right\}$, we may assume $P_0\neq O$ and $P\neq O,\pm P_0$. Then writing
    	$$P=(x,y)=(\frac{a}{d^2},\frac{b}{d^3})\qquad P_0=(x_0,y_0)=(\frac{a_0}{d_0^2},\frac{b_0}{d_0^3})$$
    	(we can write the coordinates in this form because of the form of the Weierstrass Equation) where the indicated fractions are in lowest terms. Thus we have
    	$$x(P+P_0)=(\frac{y-y_0}{x-x_0})^2-x-x_0$$.
    	Now multiplying this out and using that P and $P_0$ satisfy the Weierstrass equation yields
    	\begin{align*}
    	    x(P+P_0)&=\frac{(xx_0+A)(x+x_0)+2B-2yy_0}{(x-x_0)^2}\\
    	    &=\frac{(aa_0+Ad^2d_0^2)(ad_0^2+a_0d^2)+2Bd^4d_0^4-2bdb_0d_0}{(ad_0^2-a_0d^2)^2}
    	\end{align*}
    	In computing the height of a rational number, cancellation between numerator and denominator can only decrease the height, so we find by an easy estimation that
    	$$H(x(P+P_0))\leq C_1'max\left\{|a|^2,|d|^4,|bd|\right\}$$
    	Since $H(x(P))=max\left\{|a|,|d|^2\right\}$, and from the equation below
    	$$b^2=a^3+Aad^4+Bd^6$$
    	we can get that 
    	$$|b|\leq C_1''max\left\{|a|^{3/2}, |d|^3\right\}$$
    	which implies that 
    	$$H(x(P+P_0))\leq C_1max\left\{|a|^2,|d|^4\right\}=C_1H(x(P))$$
    	Now taking logarithms gives the desired result.
    	
    	(b) By choosing $C_2\geq 4h_x(T)$ for each of the points $T\in E(\mathbf{Q}[2])$, we may assume that $[2]P\neq O$. Then writing $P=(x,y)$, the duplication formula reads
    	$$x([2]P)=\frac{x^4-2Ax^2-8Bx+A^2}{4x^3+4Ax+4B}$$
    	It is convenient to define homogeneous polynomials
    	$$F(X,Z)=X^4-2AX^2Z^2-8BXZ^3+A^2Z^4$$
    	$$G(X,Z)=4X^3Z+4AXZ^3+4BZ^4$$
    	Then if we write $x=x(P)=a/b$ as a fraction in lowest terms, x([2]P) can be written as a quotient of integers
    	$$x([2]P)=F(a,b)/G(a,b)$$
    	Unlike what we've done in (a), we have to find a lower bound for H(x([2]P)), so it will be important to bound how much cancellation can occur between numerator and denominator.
    	The idea is to use the fact F(X,1) and G(X,1) are relative prime polynomials, so they generate the unit ideal in $\mathbf{Q}$.
    	\newtheorem{5}{Sublemma}[subsection]
    	\begin{5}
    	Let $\Delta=4A^3+27B^2$
    	\begin{align*}
    		&f_1(X,Z)=12X^2Z+16AZ^3\\
    		&g_1(X,Z)=3X^3-5AXZ^2-27BZ^3\\
    		&f_2(X,Z)=4(4A^3+27B^2)X^3-4A^2BX^2Z+4A(3A^3+22B^2)XZ^2+12B(A^3+8B^2)Z^3\\
    		&g_2(X,Z)=A^2BX^3+A(5A^3+32B^2)X^2Z+2B(13A^3+96B^2)XZ^2-3A^2(A^3+8B^2)Z^3
    	\end{align*}
    	Then the following identities hold in $\mathbf{Q}[X,Z]$:
    	$$f_1(X,Z)F(X,Z)-g_1(X,Z)G(X,Z)=4\Delta Z^7$$
    	$$f_2(X,Z)F(X,Z)-g_2(X,Z)G(X,Z)=4\Delta X^7$$
    	\end{5}
        Let
        $$\delta=gcd(F(a,b),G(a,b))$$
        be the cancellation in our fraction for x([2]P). From equations 
        $$f_1(a,b)F(a,b)-g_1(a,b)G(a,b)=4\Delta b^7$$
        $$f_2(a,b)F(a,b)-g_2(a,b)G(a,b)=4\Delta a^7$$
        we see that $\delta$ divides 4$\Delta$. Hence we obtain the bound
        $$\delta\leq |4\Delta|$$
        and so
        $$H(x([2]P))\geq max\left\{F(a,b),G(a,b)\right\}/|4\Delta|$$
        On the other hand, the same identites give the estimates
        $$|4\Delta b^7|\leq 2max\left\{f_1(a,b), g_1(a,b)\right\}max\left\{F(a,b), G(a,b)\right\}$$
        $$|4\Delta a^7|\leq 2max\left\{f_2(a,b), g_2(a,b)\right\}max\left\{F(a,b), G(a,b)\right\}$$
        Now looking at the expressions for $f_1,f_2,g_1,g_2$, we have
        $$max\left\{f_1(a,b),g_1(a,b),f_2(a,b),g_2(a,b)\right\}\geq Cmax\left\{|a|^3,|b|^3\right\}$$
        where C is a constant relying on A and B. Combining the last three inequlities yields
        $$max\left\{|4\Delta a^7|,|4\Delta b^7|\right\}\leq 2Cmax\left\{|a|^3, |b|^3\right\}max\left\{F(a,b), G(a,b)\right\}$$
        And so cancelling $max\left\{|a|^3, |b|^3\right\}$ gives
        $$max\left\{F(a,b), G(a,b)\right\}/|4\Delta|\geq (2C)^{-1}max\left\{|a|^4,|b|^4\right\}$$
        Since $max\left\{|a|^4,|b|^4\right\}=H(x(P))^4$, this gives the desired estimate
        $$H(x([2]P))\geq (2C)^{-1}H(x(P))^4$$
        and now taking logarithms gives the desired result.
        
        (c) For any constant C, the set
        $$\left\{t\in \mathbf{Q}:H(t)\leq C\right\}$$
        is obviously finite. And given any x, there will be at most two values of y satisfying the Weierstrass equation. Thus we have
        $$\left\{P\in \mathbf{Q}:h_x(P)\leq C_3\right\}$$
        is finite.
    \end{proof}
    Using the Decent theorem, the Weak Mordell-Weil theorem for $m=2$ and the lemma above, we can see that $E(\mathbf{Q})$ is finite generated.
    \subsection{Heights on Projective Space}
    We want to prove the Weak Mordell-Weil theorem for any number field K, so we have to find a height function satisfying the three properties, and then by applying the Desecent theorem we can finish the proof. However, unlike $\mathbf{Q}$, it's not easy to define a height function on other number fields. So we have to prove a lot of things to develop a height function in general cases.
    \begin{3}
    	The set of standard absolute value on $\mathbf{Q}$, which we again denote by $M_\mathbf{Q}$, consists of the following:\\
    	($\romannumeral1$) $M_\mathbf{Q}$ contains one archimedean absolute value, given by
    	$$|x|_\infty =usual\ absolute\ value$$
    	($\romannumeral2$) For each prime $p\in \mathbf{Z}$, $M+\mathbf{Q}$ contains one non-archimedean (p-adic) absolute value, given by
    	$$|p^n\frac{a}{b}|=p^{-n} \mbox{for} a,b\in \mathbf{Z},\qquad gcd(p,ab)=1$$
    	The set of standard absolute values on K, denoted $M_K$, consists of all absolute values on K whose restriction to $\mathbf{Q}$ is one of the absolute values in $M_\mathbf{Q}$.
    \end{3}
    \begin{3}
	    For $v\in M_K$, the local degree at v, denoted $n_v$, is given by
	    $$n_v=[K_v:\mathbf{Q}_v]$$
	    Here $K_v$ and $\mathbf{Q}_v$ denote the completion of the field with respect to the absolute value v.
    \end{3}
    With these definitions, we can state two basic facts from algebraic number theory which will be needed.
    \begin{4}
    	Let $L/K/\mathbf{Q}$ be a tower of number fields, and $v\in M_K$. Then
    	\begin{equation}
    		\sum_{w\in M_L\atop w|v}n_w=[L:K]n_v
    	\end{equation}
    \end{4}
    \begin{4}
    	Let $x\in K^*$. Then
    	$$\prod_{v\in M_K}|x|^{n_v}=1$$
    \end{4}
    Next we will define the height of a point in projective space.
    \begin{3}
    	Let $P\in \mathbf{P}^N(K)$ be a point with homogeneous coordinates
    	$$P=[x_0,\dots,x_N],\qquad x_i\in K$$
    	The height of P (relative to K) is defined by
    	$$H_K(P)=\prod_{v\in M_K}max\left\{|x_0|_v,\dots |x_N|_v\right\}^{n_v}$$
    \end{3}
    As we can see, when $K=\mathbf{Q}$, this definition is the same as 
    $$H(P)=max \left\{|x_0|,\dots, |x_N|\right\}$$
    where 
    $$x_0,\dots ,x_N\in \mathbf{Z}\quad and\quad gcd(x_0,\dots, x_N)=1.$$
    We will state some important proerties of the given height function.
    \begin{4}
    	Let $P\in \mathbf{P}^N(K)$\\
    	(a) The height $H_K(P)$ does not depend on the choice of homogeneous coordinates for P.\\
    	(b) $H_K(P)\geq 1$\\
    	(c) Let L/K be a finite extension. Then 
    	$$H_L(P)=H_K(P)^{[L:K]}$$
    \end{4}
    \begin{proof}
    	(a) It is directly from Proposition 2.3.2.\\
    	(b) For any point in projective space, one can find homogeneous coordinates by multiplying a number so that one of the coordinates is 1. Then every factor in the product defining $H_K(P)$ is at least 1.\\
    	(c) We compute
    	\begin{align*}
    	    H_L(P)&=\prod_{w\in M_L}max\left\{|x_0|_w,\dots |x_N|_w\right\}^{n_w}\\
    	    &=\prod_{v\in M_K}\prod_{w\in M_L\atop w|v}max\left\{|x_0|_v,\dots |x_N|_v\right\}^{n_w}\qquad since\ x_i\in K\\
    	    &=\prod_{v\in M_K}max\left\{|x_0|_v,\dots |x_N|_v\right\}^{[L:K]n_v}\\
    	    &=H_K(P)^{[L:K]}
    	\end{align*}
    \end{proof}
    Sometimes, when a field is not given, it's easier to use a height function not relative to a field.
    \begin{3}
    	Let $P\in \mathbf{P}^N(\bar{\mathbf{Q}})$. The absolute height of P, denoted H(P), is defined as follows. Choose any field K such that $P\in \mathbf{P}^N(K)$. Then
    	$$H(P)=H_K(P)^{1/[K:\mathbf{Q}]}$$
    	In view of Proposition 2.3.3, it's easy to see that this is well defined.
    \end{3}
    We now investigate how the height changes under mappings between projective spaces.
    \begin{3}
    	A morphism of degree d between projective spaces is a map
    	$$F:\mathbf{P}^N\rightarrow \mathbf{P}^M$$
    	$$F(P)=[f_0(P),\dots, f^M(P)]$$
    	where $f_0,\dots , f_M\in \bar{\mathbf{Q}}[X_0,\dots, X_N]$ are homogeneous polynomials of degree d with no commone zero in $\bar{\mathbf{Q}}$ other than $X_0=\dots =X_N=0$.
    \end{3}
    To prove the height function has the three properties, we have to find the lower bound and upper bound of the height function. Therefore we have the following theorem:
    \begin{1}
    	Let
    	$$F:\mathbf{P}^N\rightarrow \mathbf{P}^M$$
    	be a morphism of degree d. Then there are constants $C_1$ and $C_2$, depending on F, so that for all points $P\in \mathbf{P}^N(\bar{\mathbf{Q}})$,
    	$$C_1H(P)^d\leq H(F(P))\leq C_2H(P)^d$$
    \end{1}
    \begin{proof}
    	Write $F=[f_0,\dots,f_M]$ with homogeneous polynomials $f_i$, and let $P=[x_0,\dots, x_N]\in \mathbf{P}^N(\bar{\mathbf{Q}})$. Choose some number field K containing $x_0,\dots,x_N$ and all of the coefficients of all of the $f_i's$. Then for each $v\in M_K$, let
    	$$|P|_v=max_{0\leq i\leq N}\left\{|x_i|_v\right\},\qquad |F(P)|_v=max_{0\leq j\leq M}\left\{|f_j(P)|_v\right\}$$
    	and
    	$$|F|_v=max\left\{|a|_v:a\ is\ a\ coefficient\ of\ some\ f_i\right\}$$
    	Then from the definition of height,
    	$$H_K(P)=\prod_{v\in M_K}|P|_v^{n_v}\qquad and\qquad H_K(F(P))=\prod_{v\in M_K}|F(P)|_v^{n_v}$$
    	so it makes sense to define
    	$$H_K(F)=\prod_{v\in M_K}|F|_v^{n_v}$$
    	Finally, we let $C_1,/dots,$ denote constants which depend only on M,N and d, and set
    	$$\varepsilon(v) = \left\{ \begin{array}{rcl}
    	1 & \mbox{if}
    	& v\in M_K^\infty \\ 0 & \mbox{if} & v\in M_K^0
    	\end{array}\right.$$
    	Having set notation, we turn to the proof of the theorem. The upper bound is relatively easy. Let $v\in M_K$. The triangle inequality yields
    	$$|f_i(P)|_v\leq C_1^{\varepsilon(v)}|F|_v|P|_v^d$$
    	Now raise to the $n_v-power$, multiply over all $v_in M_K$, and take the $[K:\mathbf{Q}]^{th}-root$. This yields the desired upper bound
    	\begin{align*}
    		H(F(P))&\leq C_1^{
    			\Sigma_{v\in M_K}\varepsilon(v)n_v
    			/[K:\mathbf{Q}]}H(F)H(P)^d\\
    	    &= C_1^{
    	    	\Sigma_{v\in M_K^\infty}n_v
    	    	/[K:\mathbf{Q}]}H(F)H(P)^d\\
    		&=C_1H(F)H(P)^d
    	\end{align*}
    	Notice that we don't use the fact that the $f_i's$ have no common non-trivial zero. But for the lower bound, we have to use this condition.
    	
    	From the Nullstellensatz theorem that the ideal generated by $f_0,\dots, f_M$ in $\bar{mathbf{Q}}[X_0,\dots, X_N]$ contains some power of each $X_0,\dots, X_N$, since each $(0,\dots, 0)$. Thus for an approriate integer $e\geq 1$, there are polynomials $g_{ij}\in \bar{mathbf{Q}}[X_0,\dots, X_N]$ such that
    	\begin{equation}
    		X_i^e=\sum_{j=0}^{M} g_{ij}f_j \qquad\mbox{for each}   \ \ 0\leq i\leq N
    	\end{equation}
    	Replacing K by a finite extension, we may assume that each $g_{ij}\in K[X_0,\dots, X_N]$. Further, by discarding all terms except those which are homogeneous of degree e, we may assume that each $g_{ij}$ is homogeneous of degree e-d. Let us set the further reasonable notation
    	$$|G|_v=max\left\{|b|_v:b\ is\ a\ coefficient\ of\ some\ g_{ij}\right\}$$
    	$$H_K(G)=\prod_{v\in M_K}|G|_v^{n_v}$$
    	Recalling that $P=[x_0,\dots, x_N]$, the equation described above imply that for each i,
    	\begin{align*}
    		|x_i|_v^e&=|\sum_{j=0}^{M}g_{ij}(P)f_j(P)|_v\\
    		&\leq C_2^{\varepsilon(v)}max_{0\leq j\leq M}\left\{|g_{ij}(P)f_j(P)|_v\right\}
    	\end{align*}
    	Now taking the maximum over i gives
    	$$|P|_v^e\leq C_2^{\varepsilon(v)}max_{0\leq j\leq M\atop 0\leq i\leq N}\left\{|g_{ij}(P)|_v\right\}|F(P)|_v$$
    	But since each $g_{ij}$ has degree e-d, the usual application of the triangle inequality yields
    	$$|g_{ij}(P)|_v\leq C_3^{\varepsilon(v)} |G|_v|P|_v^{e-d}$$
    	Substituting this in above and multiplying through by $|P|_v^{d-e}$ gives
    	$$|P|_v^d\leq C_4^{\varepsilon(v)}|G|_v|F(P)|_v$$
    	and now the usual raising to the $n_v$-power, multiplying over $v\in M_K$ and taking the $[K:\mathbf{Q}]^{th}$-root yields the desired lower bound.
    \end{proof}
    \begin{3}
    	For $x\in \bar{\mathbf{Q}}$, let
    	$$H(x)=H([x,1])$$
    	Similarly, if $x\in K$, then
    	$$H_K(x)=H_K([x,1])$$
    \end{3}
    \begin{1}
    	Let
    	$$f(T)=a_0T^d+a_1T^{d-1}+\dots+a_d=a_0(T-\alpha_1)\dots (T-\alpha_d)\in \bar{\mathbf{Q}}[T]$$
    	be a polynomial of degree d. Then
    	$$2^{-d}\prod_{j=1}^{d}H(\alpha_j)\leq H([a_0,\dots,a_d])\leq 2^{d-1}\prod_{j=1}^{d}H(\alpha_j)$$
    \end{1}
    \begin{proof}
    	First note that the inequality remains unchanged if $f(T)$ is replaced by $(1/a_0)f(T)$. Thus we can prove the inequality under the condition of $a_0=1$.
    	
    	Let $\mathbf{Q}(\alpha_1,\dots,\alpha_d)$, and for $v\in M_K$, set
    	$$\varepsilon(v) = \left\{ \begin{array}{rcl}
    	2 & \mbox{if}
    	& v\in M_K^\infty \\ 1 & \mbox{if} & v\in M_K^0
    	\end{array}\right.$$
    	(Note that this function is different from the $\varepsilon(v)$ we defined before, that is because we can see that
    	$$|x+y|_v\leq \varepsilon(v)max\left\{|x|_v,|y|_v\right\}$$
    	which will be helpful to solve the inequality, due to its form.)
    	
    	We will now prove that
    	$$\varepsilon(v)^{-d}\prod_{j=1}^{d}max\left\{|\alpha_j|_v,1\right\}\leq max_{0\leq i\leq d}\left\{|a_i|_v\right\}\leq \varepsilon(v)^{d-1}\prod_{j=1}^{d}max_{0\leq i\leq d}\left\{|a_i|_v,1\right\}$$
    	Once this is done, raising to the $n_v$-power, multiplying over $v\in M_K$, and taking $[K:\mathbf{Q}]^{th}$-roots gives the desired result.
    	
    	The proof is by induction of $d=deg(f)$. For $d=1$, the inequality is clear. Assume for polynomials of degree d-1, the result is true. Choose an index k so that
    	$$|\alpha_k|_v\geq |\alpha_j|_v \qquad \mbox{for all} \ 0\leq j\leq d$$ 
    	And we define a polynomial
    	\begin{align*}
    		g(T)&=(T-\alpha_1)\dots(T-\alpha_{k-1})(T-\alpha_{k+1})\dots(T-\alpha_d)\\
    		&=b_0T^{d-1}+\dots+b_{d-1}
    	\end{align*}
    	Therefore we can see that $f(T)=(T-\alpha_k)g(T)$. By comparing coefficients we can get
    	$$a_i=b_i-\alpha_kb_{i-1}$$
    	We now prove the upper above bound.
    	\begin{align*}
    		max_{0\leq i\leq d}\left\{|a_i|_v\right\}&=max_{0\leq i\leq d}\left\{|b_i-\alpha_kb_{i-1}|\right\}\\
    		&\leq \varepsilon(v)max_{0\leq i\leq d}\left\{|b_i|,|\alpha_kb_{i-1}|\right\}\\
    		&\leq \varepsilon(v)max_{0\leq i\leq d}\left\{|b_i|\right\}max\left\{|\alpha_k|,1\right\}\\
    		&\leq\varepsilon(v)^{d-1}\prod_{j=1}^{d}max_{0\leq i\leq d}\left\{|a_i|_v,1\right\}
    	\end{align*}
    	(The last step is by the induction hypothesis applied to g).
    	
    	Next, to prove the lower bound, we consider two cases. First, if $|\alpha_k|_v\leq \varepsilon(v)$, then by the choice of the index k,
    	$$\prod_{j=1}^{d}max\left\{|\alpha_j|_v,1\right\}\leq max\left\{|\alpha_k|_v\right\}\leq \varepsilon(v)^d$$
    	And remember that $a_0=1$, so we have
    	$$max_{0\leq i\leq d}\left\{|a_i|_v\right\}\geq 1$$
    	Therefore
    	$$\varepsilon(v)^{-d}\prod_{j=1}^{d}max\left\{|\alpha_j|_v,1\right\}\leq max_{0\leq i\leq d}\left\{|a_i|_v\right\}$$
    	Second, if $|\alpha_k|_v\geq \varepsilon(v)$, then
    	\begin{align*}
    		max_{0\leq i\leq d}\left\{|a_i|_v\right\}&=max_{0\leq i\leq d}\left\{|b_i-\alpha_kb_{i-1}|_v\right\}\\
    		&=max_{0\leq i\leq d-1}\left\{|b_i|_v\right\}|\alpha_k|_v
    	\end{align*}
    	for $v\in M_K^0$. And for $v\in M_K^\infty$,
    	\begin{align*}
    	max_{0\leq i\leq d}\left\{|b_i-\alpha_kb_{i-1}\right\}&\geq (|\alpha_k|_v-1)max_{0\leq i\leq d-1}\left\{|b_i|_v\right\}\\
    	&>\varepsilon(v)^{-1}|\alpha_k|_vmax_{0\leq i\leq d-1}\left\{|b_i|_v\right\}
    	\end{align*}
    	Combining the two situations we can get that
    	$$max_{0\leq i\leq d}\left\{|a_i|_v\right\}\geq \varepsilon(v)^{-1}|\alpha_k|_vmax_{0\leq i\leq d-1}\left\{|b_i|_v\right\}$$
    	And now applying the induction hypothesis to g gives the desired lower bound, which completes the proof.
    \end{proof}
    The reason we prove this theorem is that we want to prove that this height function 'satisfies' the third property. But before we do that, we have to prove another lemma.
    \begin{2}
    	Let $P\in \mathbf{P}^N(\bar{\mathbf{Q}})$ and $\sigma\in G_{\bar{\mathbf{Q}}/\mathbf{Q}}$. Then
    	$$H(P^\sigma)=H(P)$$
    \end{2}
    \begin{proof}
    	Let $K/\mathbf{Q}$ be a field with $P\in \mathbf{P}^N(K)$ and $\sigma$ gives an isomorphism $\sigma:K\rightarrow K^\sigma$. It likewise identifies the sets of absolute values,
    	$$\sigma: M_K\rightarrow M_{K^\sigma}$$
    	$$v\rightarrow v^\sigma$$
    	Clearly $\sigma$ also gives an isomorphism $K_v\rightarrow K_{v^\sigma}^\sigma$, so $n_v=n_{v^\sigma}$. We now compute
    	\begin{align*}
    		H_{K^\sigma}(P^\sigma)&=\prod_{w\in M_{K^\sigma}}max\left\{|x_i^\sigma|_w\right\}^{n_w}\\
    		&=\prod_{v\in M_K}max\left\{|x_i^\sigma|_{v^\sigma}\right\}^{n_{v^\sigma}}\\
    		&=\prod_{v\in M_K}max\left\{|x_i|_v\right\}^{n_v}\\
    		&=H_K(P).
    	\end{align*}
    	Since $[K:\mathbf{Q}]=[K^\sigma:\mathbf{Q}]$, this is the desired result.
    \end{proof}
    \begin{1}
    	Let C and d be constants. Then the set 
    	$$\left\{P\in\mathbf{P}^N(\bar{\mathbf{Q}}):H(P)\leq C\ and\ [\mathbf{Q}(P):\mathbf{Q}]\leq d\right\}$$
    	contains only finitely many points. In particular, for any number field K,
    	$$\left\{P\in \mathbf{P}^N(K):H_K(P)\leq C\right\}$$
    	is a finite set.
    \end{1}
    \begin{proof}
    	Let $P\in \mathbf{P}^N(\bar{\mathbf{Q}})$. Take homogeneous coordinates for P, say
    	$$P=[x_0,\dots,x_N]$$
    	with some $x_j=1$. Then $\mathbf{Q}(P)=\mathbf{Q}(x_0,\dots,x_N)$, and we have the easy estimate
    	\begin{align*}
    		H_{\mathbf{Q}(P)}(P)&=\prod_{v\in M_{\mathbf{Q}(P)}}max_{0\leq i\leq N}\left\{|x_i|_v\right\}^{n_v}\\
    		&\geq max_{0\leq i\leq N}(\prod_{v\in M_{\mathbf{Q}(P)}}max\left\{|x_i|_v,1\right\}^{n_v})\\
    		&=max_{0\leq i\leq N}H_{\mathbf{Q}(P)}(x_i)
    	\end{align*}
    	Thus if $H(P)\leq C$ and $[\mathbf{Q}(P):\mathbf{Q}]\leq d$, then
    	$$max_{0\leq i\leq N}\ H(x_i)\leq C\qquad max_{0\leq i\leq N}\ [\mathbf{Q}(x_i):\mathbf{Q}]\leq d$$
    	It thus suffices to prove that the set
    	$$\left\{x\in \bar{\mathbf{Q}}:H(x)\leq C\ and\ [\mathbf{Q}(x_i):\mathbf{Q}]\leq d\right\}$$
    	is finite, which means that we only have to prove the case $N=1$.
    	
    	Suppose $x\in\bar{\mathbf{Q}}$ is in this set, and let $e=[\mathbf{Q}(x):\mathbf{Q}]\leq d$. Further let $x=x_1,\dots, x_e$ be the conjugates of x, so the minimal polynomial of x over $\mathbf{Q}$ is
    	$$f_x(T)=(T-x_1)\dots(T-x_e)=T^e+\dots+a_e\in \mathbf{Q}(T)$$
    	Now
    	\begin{align*}
    		H([1,a_1,\dots,a_e])&\leq 2^{e-1}\prod_{j=1}^{e}H(x_j)\\
    		&=2^{e-1}H(x)^e\\
    		&\leq (2C)^d
    	\end{align*}
    	Thus we can see that there are only finitely many choices for $a_i$, so the set is finite, which completes the proof.
    \end{proof}
    \subsection{Heights on Elliptic Curves}
    Now we have already developed enough theorems about height functions on projective space, so we will focus on elliptic curves, and finish the proof of the Weak-Mordell theorem.
    \begin{3}
    	Let f,g be two real-valued functions on a set $\Phi$. Then we write 
    	$$f=g+O(1)$$
    	if there's constants $C_1$ and $C_2$ so that
    	$$C_1\leq f(P)-g(P)\leq C_2\ \mbox{for all}\ P\in \Phi$$ 
    	In only the lower(respectively upper) inequality is satisfied, then we naturally write $f\geq g+O(1)$(respecively $f\leq g+O(1)$).
    \end{3}
    \begin{3}
    	The height on projective space is the function
    	$$h:\mathbf{P}^N(\bar{\mathbf{Q}})\rightarrow \mathbf{R}$$
    	$$h(P)=logH(P)$$
    	Note that $h(P)\geq 0$ for all P since $H(P)\geq1$.
    	
    	And let E/K be an elliptic curve and $f\in \bar{K}(E)$ a function. The height on E(relative to f) is the function
    	$$h_f:E(\bar{K})\rightarrow \mathbf{R}$$
    	$$h_f(P)=h(f(P))$$
    \end{3}
    \begin{4}
    	Let E/K be an elliptic curve and $f\in K(E)$ a non-constant function. The for any constant C,
    	$$\left\{P\in E(K):h_f(P)\leq C\right\}$$
    	is a finite set.
    \end{4}
    \begin{proof}
    	The function f gives a finite-to-one map of the set in question to the set
    	$$\left\{Q\in \mathbf{P}^1(K):H(Q)\leq e^C\right\}$$
    	Now apply Theorem to this last set and we can get the desired result.
    \end{proof}
    The next step helps us find the relation between the additive law on an elliptic curve and the height function.
    \begin{1}
    	Let E/K be an elliptic curve and let $f\in K(E)$ be an even function(i.e. $f\circ[-1]=f$). Then for all $P,Q\in E(\bar{K})$,
    	$$h_f(P+Q)+h_f(P-Q)=2h_f(P)+2h_f(Q)+O(1)$$
    \end{1} 
    \begin{proof}
    	Choose a Weierstrass equation for E/K of the form
    	$$E:y^2=x^3+Ax+B$$
    	We start by proving the theorem for the particular function $f=x$(Note that it is an even function). The general case will be an easy corollary.
    	
    	Since $h_x(O)=0$ and $h_x(-P)=h_x(P)$, the result clearly holds if $P=O$ or $Q=O$. We now assume that $P,Q\neq O$, and write
    	$$x(P)=[x_1,1]\qquad x(Q)=[x_2,1]$$
    	$$x(P+Q)=[x_3,1]\qquad x(P-Q)=[x_4,1]$$
    	Thus we have
    	$$x_3+x_4=\frac{2(x_1+x_2)(A+x_1x_2)+4B}{(x_1+x_2)^2-4x_1x_2}$$
    	$$x_3x_4=\frac{(x_1x_2-A)^2-4B(x_1+x_2)}{(x_1+x_2)^2-4x_1x_2}$$
    	Define a map $g:\mathbf{P}^2\rightarrow\mathbf{P}^2$ by
    	$$g([t,u,v])=[u^2-4tv, 2u(At+v),(v-At)^2-4Btu]$$
    	Then the formula for $x_3$ and $x_4$ shows that there is a commutative diagram
    	\begin{align*}
    		\qquad E\ &\times \ E\xrightarrow{G}E\ \times\ E\\
    		&\downarrow\qquad \qquad \quad \downarrow\\
    		\sigma\qquad \mathbf{P}^1\ &\times\ \mathbf{P}^1\quad \mathbf{P}^1\ \times\ \mathbf{P}^1\quad \sigma\\
    		&\downarrow\qquad \qquad \quad \downarrow\\
    		&\mathbf{P}^2\ \quad \xrightarrow{g}\ \quad \mathbf{P}^2
    	\end{align*}
    	where
    	$$G(P,Q)=(P+Q, P-Q)$$
    	and the vertical map $\sigma$ is the composition of the two maps
    	$$E\times E\rightarrow \mathbf{P}^1\times\mathbf{P}^1\qquad and\qquad \mathbf{P}^1\times\mathbf{P}^1\rightarrow \mathbf{P}^2$$
    	$$(P,Q)\rightarrow(x(P),x(Q))\qquad ([\alpha_1,\beta_1],[\alpha_2,\beta_2])\rightarrow [\beta_1\beta_2, \alpha_1\beta_2+\alpha_2\beta_1,\alpha_1\alpha_2]$$
    \end{proof}
    The next step is to show that g is morphism, so as to be able to apply Theorem. By definition, this is equivalent to prove that the three polynomials have no common non-trivial zeros. Suppose that $g([t,u,v])=[0,0,0]$. If $t=0$, then from
    $$u^2-4tv=0\qquad and\qquad (v-At)^2-4Btu=0$$
    we see that $u=v=0$. Thus we may assume that $t\neq 0$, and so it makes sense to define a new quantity $x=u/2t$. Notice that the equation $u^2-4tv=0$ can be written as $x^2=v/t$. Now dividing the equalities
    $$2u(At+v)+4Bt^2=0\qquad and\qquad (v-At)^2-4Btu=0$$
    by $t^2$ and rewriting them interms of x yields the two equations
    $$\psi(x)=4x^3+4Ax+4B=0$$
    $$\phi(x)=x^4-2Ax^2-8Bx+A^2=0$$
    And we can see that
    $$(12X^2+16A)\phi(X)-(3X^3-5AX-27B)\psi(X)=4(4A^3+27B^2)\neq 0$$
    This completes the proof that g is a morphism.
    
    We return to our commutative diagram, and compute
    \begin{align*}
    	h(\sigma(P+Q,P-Q))&=h(\sigma\circ G(P,Q))\\
    	&=h(g\circ \sigma(P,Q))\\
    	&=2h(\sigma(P,Q))+O(1)\qquad from\ Theorem 
    \end{align*}
    since g is a morphism of degree 2. Now to complete the proof for $f=x$, we will show that for all $R_1,R_2\in E(\bar{K})$, there is a relation
    $$h(\sigma(R_1,R_2))=h_x(R_1)+h_x(R_2)+O(1)$$
    Then using this twice on both sides of the equation
    $$h(\sigma(P+Q,P-Q))=2h(\sigma(P,Q))+O(1)$$
    will give the desired result.
    
    One verifies that if either $R_1=O$ or $R_2=O$, then $h(\sigma(R_1,R_2))=h_x(R_1)+h_x(R_2)$. Otherwise, we may write
    $$x(R_1)=[\alpha_1,1]\qquad and\qquad x(R_2)=[\alpha_2,1]$$
    and so
    $$h(\sigma(R_1,R_2))=h([1,\alpha_1+\alpha_2,\alpha_1\alpha_2])\quad and\quad h_x(R_1)+h_x(R_2)=h(\alpha_1)+h(\alpha_2)$$
    Then from Theorem applied to the polynomial $(T+\alpha_1)(T+\alpha_2)$, we obtain the desired estimate
    $$h(\alpha_1)+h(\alpha_2)-log4\leq h([1,\alpha_1+\alpha_2,\alpha_1\alpha_2])\leq h(\alpha_1)+h(\alpha_2)+log2$$
    Finally, to deal with the general case, we prove that
    $$h_f=\frac{deg(f)h_x}{2}+O(1)$$
    Once this is proved, the theorem follows immediately from multiplying the known relation for $h_x$ by $\frac{deg(f)}{2}$
    \begin{2}
    	Let $f,g\in K(E)$ be even functions. Then
    	$$(deg\ g)h_f=(deg\ f)h_g+O(1)$$
    \end{2}
    \begin{proof}
    	Let $x,y\in K(E)$ be Weierstrass coordinates for E/K. The subfield consisting of all even functions is exactly K(x), so we can find a rational function $\rho(X)\in K(X)$ so that
    	$$f=\rho\circ x$$
    	Hence using theorem and the fact that $\rho$ is a morphism, we can get that
    	$$h_f=(deg\ \rho)h_x+O(1)$$
    	But from the equation above, we have
    	$$degf=degxdeg\rho=2deg\rho$$
    	So we find 
    	$$2h_f=(deg\ f)h_x+O(1)$$
    	The same reasoning for g yields
    	$$2h_g=(deg\ g)h_x+O(1)$$
    	and combining the last two equations gives the desired result.
    	
    \end{proof}
    \newtheorem{6}{Corollary}
    \begin{6}
    	Let E/K be an elliptic curve and $f\in E(K)$ an even function. \\
    	(a) Let $Q\in E(\bar{K})$. Then for all $P\in E(\bar{K})$,
    	$$h_f(P+Q)\leq 2h_f(P)+O(1)$$
    	(b) Let $m\in \mathbf{Z}$. Then for all $P\in E(\bar{K})$,
    	$$h_f([m]P)=m^2h_f(P)+O(1)$$
    \end{6}
    \begin{proof}
    	(a) This is directly from Theorem 2.4.2, because $h_f(P-Q)\geq 0$.\\
    	(b) Since f is even, it suffices to consider $m\geq 0$. Further, this result is trivial for $m=0,1$. We will finish the proof by induction. Assume it is known for m-1 and m. Replacing P,Q in Theorem 2.4.3 by [m]P and P, we find
    	\begin{align*}
    		h_f([m+1]P)&=-h_f([m-1]P)+2h_f([m]P)+2h_f(P)+O(1)\\
    		&=(-(m-1)^2+2m^2+2)h_f(P)+O(1)\\
    		&=(m+1)^2h_f(P)+O(1)
    	\end{align*}
    \end{proof}
    \begin{1}
    	\textup{(Mordell-Weil)} Let E be an elliptic curve defined over a number field K. The group E(K) is a finitely generated Abelian group
    \end{1}
    \begin{proof}
	    Choose any even, non-constant function $f\in K(E)$, for example the x-coordinate function on Weierstrass equation. Then from Corollary 2.4.4(a), 2.4.4(b) and Proposition 2.4.1 we know that the height function $h_f$ satisfies the three properties, and using the Descent theorem and Weak Mordell-Weil theorem we know Mordell-Weil theorem is correct.
    \end{proof}
    

\begin{thebibliography}{4}
    	\bibitem{ref1}Silverman J H . The Arithmetic of Elliptic Curves[J]. Inventiones Mathematicae, 1974, 23(3-4):179-206.
    	\bibitem{ref2}J.S. Milne. Elliptic Curves
    	\bibitem{ref3}Hartshorne R . Algebraic Geometry[M]. American Mathematical So, 1975.
    	\bibitem{ref4} Fulton. An Introduction to Algebraic Geometry 
    \end{thebibliography}
\end{document}